\documentclass[a4paper, 10pt]{amsart}

\usepackage{amsthm, amsmath, amssymb}
\usepackage{xcolor}
\usepackage{enumitem}

\usepackage[noadjust]{cite}
\usepackage[colorlinks, citecolor=red, urlcolor=blue, bookmarks=false, hypertexnames=true]{hyperref}

\theoremstyle{plain}
\newtheorem{theorem}{Theorem}[section]
\newtheorem*{theorem*}{Theorem \ref{thm:appl}}
\newtheorem{proposition}[theorem]{Proposition}

\newtheorem{lemma}[theorem]{Lemma}

\theoremstyle{definition}

\newtheorem{remark}[theorem]{Remark}

\numberwithin{equation}{section}

\DeclareMathOperator{\trace}{trace}
\DeclareMathOperator{\grad}{grad}

\DeclareMathOperator{\Div}{div}
\DeclareMathOperator{\id}{Id}
\DeclareMathOperator{\Hess}{Hess}
\DeclareMathOperator{\Riem}{Riem}

\allowdisplaybreaks

\title[Rigidity results for biconservative hypersurfaces]{Rigidity results for compact biconservative hypersurfaces in space forms}
\author{\c{S}tefan Andronic and Aykut Kayhan}
\address{Faculty of Mathematics, Al. I. Cuza University of Iasi, Blvd. Carol I, no. 11, 700506 Iasi, Romania} \email{stefanandronic215@gmail.com}
\address{Permanent address: Mathematics and Science Education, Maltepe University, Istanbul, 34480, Turkey}
\address{Current address: Faculty of Mathematics, Al. I. Cuza University of Iasi, Blvd. Carol I, no. 11, 700506 Iasi, Romania} \email{aykutkayhan@maltepe.edu.tr}
	

\thanks{The second author was supported by the Scientific and Technological Research Council of T\"{u}rkiye (T\"{U}B\.{I}TAK), under the 2219-International Post-Doctoral Research Fellowship Programme (grant no: 1059B192300478). The opinions and views expressed herein are those of the authors and do not reflect those of T\"{U}B\.{I}TAK}

\begin{document}
	\maketitle
	
	\begin{abstract}
		In this paper we present alternative proofs for two known rigidity results concerning non-negatively curved compact biconservative hypersurfaces in space forms. Further, we prove some new rigidity results by replacing the hypothesis of non-negative sectional curvature with some estimates of the squared norm of the shape operator.
	\end{abstract}
	
	\section{Introduction}
	
	The investigation of biconservative submanifolds has seen a significant advancement in recent years, evolving from the study of biharmonic submanifolds as researchers sought to relax the biharmonic equation. We just recall here that the biharmonic equation of isometric immersions $\varphi : M^m \to N^n$ is 
	$$
	\tau_2 (\varphi) = - \Delta^\varphi \tau (\varphi) - \trace R^N (d\varphi (\cdot), \tau (\varphi)) d \varphi (\cdot) = 0,
	$$
	where $\Delta^\varphi$ is the rough Laplacian acting on sections of the pull-back bundle $\varphi ^{-1} \left ( TN^n \right )$, $R^N$ is the curvature tensor field on $N^n$, $\tau (\varphi) = m H$ is the tension field associated to $\varphi$ and $H$ is the mean curvature vector field. The biharmonic equation decomposes into the normal and tangent parts. The biconservative submanifolds are characterized by
	\begin{equation} \label{eq:tau2tangentVanishing}
		(\tau_2 (\varphi))^\top = 0.
	\end{equation}
	For a geometrical meaning of this equation see \cite{CaddeoMontaldoOniciucPiu2014}, \cite{LoubeauMontaldoOniciuc2008} and for recent surveys on this topic see, for example, \cite{FetcuOniciuc2022} and \cite{FuZhan2023}. In addition to the significantly greater number of examples of biconservative submanifolds compared to biharmonic submanifolds, biconservative submanifolds also exhibit numerous intriguing geometric properties, particularly in the context of surfaces (see, for example, \cite{Nistor2017}). In the case of biconservative hypersurfaces in space forms, we recall that they are characterized by
	\begin{equation}\label{eq:BiconservativityCharacterization}
		A(\grad f) = - \frac m 2 f \grad f.
	\end{equation}
	and they satisfy the following inequality
	\begin{equation}\label{eq:JHChenInequality}
		|\nabla A|^2 \geq \frac {m^2 (m+26)} {4(m-1)} |\grad f|^2,
	\end{equation}
	where  $A$ denotes the shape operator and $f$ is the mean curvature function of $M^m$ (see \cite{BalmusMontaldoOniciuc2012} and \cite{JHChen1993}).
	
	Obviously, every hypersurface with constant mean curvature function (CMC) is also biconservative. For non-CMC biconservative hypersurfaces in space forms, most of the existing results pertain to local classifications (see, for example, \cite{HasanisVlachos1995}), and only a few results are known for compact non-CMC biconservative hypersurfaces, see \cite{MontaldoOniciucPampano2023} and \cite{MontaldoPampano2023}.
	
	In the literature of hypersurfaces, there are two well-known results that classify compact hypersurfaces in spaces forms. The first one is due to Nomizu and Smyth in \cite{NomizuSmyth1969} and it states
	\begin{theorem} \label{th:NomizuSmythTheorem}
		Let $\varphi : M^m \to N^{m+1} (c)$ be a compact hypersurface in a space form. If $M$ has non-negative sectional curvature, i.e. $\Riem ^M \geq 0$, and constant mean curvature, then $\nabla A = 0$ and $\varphi (M)$ is one of the following hypersurfaces
		\begin{enumerate}[label = \alph*)]
			\item the Euclidean hypersphere $\mathbb S^m (r)$ of radius $r > 0$, if $c \in \{-1, 0\}$, i.e. $N$ is either the hyperbolic space $\mathbb H^{m+1}$ or the Euclidean space $\mathbb R^{m+1}$;
			\item either the small hypersphere $\mathbb S^m (r)$, $r \in (0, 1)$, or the standard product $\mathbb S^{m_1} (r_1) \times \mathbb S^{m_2} (r_2)$, where $r_1^2 + r_2^2 = 1$, $m_1 + m_2 = m$.
		\end{enumerate}
		Moreover, in the cases a) and b), for $m_1 \geq 2$, $m_2 \geq 2$, the immersion $\varphi$ has to be an embedding.
	\end{theorem} 
	
	The second one was given by Cheng and Yau in \cite{ChengYau1977} and basically leads to the same conclusion, but the hypothesis of having constant mean curvature is replaced by constant normalized scalar curvature greater or equal to $c$. 
	
	The biconservative hypersurfaces represent a natural notion in the geometry of submanifolds because they allow the generalization of these two results. Indeed, if one replaces the hypothesis of CMC with biconservativity and $m \leq 10$ in Theorem \ref{th:NomizuSmythTheorem}, then the following result is obtained.  
	\begin{theorem}[\cite{FetcuOniciuc2022}] \label{th:FetcuOniciuc2022}
		Let $\varphi : M^m \to N^{m+1} (c)$ be a compact non-minimal biconservative hypersurface. If $\Riem^M \geq 0$ and $m \leq 10$, then $\nabla A = 0$ and $\varphi (M)$ is one of the following hypersurfaces
		\begin{enumerate}[label = \alph*)]
			\item the Euclidean hypersphere $\mathbb S^m (r)$ of radius $r > 0$, if $c \in \{-1, 0\}$, i.e. $N$ is either the hyperbolic space $\mathbb H^{m+1}$ or the Euclidean space $\mathbb R^{m+1}$;
			\item either the small hypersphere $\mathbb S^m (r)$, $r \in (0, 1)$, or the standard product $\mathbb S^{m_1} (r_1) \times \mathbb S^{m_2} (r_2)$, where $r_1^2 + r_2^2 = 1$, $m_1 + m_2 = m$ and $r_1 \neq \sqrt {m_1 / m}$, if $c = 1$, i.e. $N$ is the unit Euclidean sphere $\mathbb S^{m+1}$.
		\end{enumerate}
	\end{theorem}
	
	If the hypothesis of normalized scalar curvature greater or equal to $c$ is replaced by the biconservativity in the theorem from \cite{ChengYau1977}, then it is obtained
	\begin{theorem} [\cite{FetcuLoubeauOniciuc2021}] \label{th:FetcuLoubeauOniciucConstantScalarCurvature}
		Let $\varphi : M^m \to N^{m+1} (c)$ be a compact biconservative hypersurface in a space form $N^{m+1} (c)$. If $M$ is not minimal, has constant normalized scalar curvature, and $\Riem ^M \geq 0$, them $M$ is 
		\begin{enumerate} [label = \alph*)]
			\item a hypersphere $\mathbb S^m (r)$, $r > 0$, if $c \in \{-1, 0\}$, i.e. $N$ is either the hyperbolic space $\mathbb H^{m+1}$, or the Euclidean space $\mathbb E^{m+1}$,
			\item either a small hypersphere $\mathbb S^m (r)$, $r \in (0, 1)$, or the standard product $\mathbb S^{m_1} (r_1) \times \mathbb S^{m_2} (r_2)$, where $r_1^2 + r_2^2 = 1$, $m_1 + m_2 = m$, and $r_1 \neq \sqrt {m_1 / m}$, if $c = 1$, i.e. $N$ is the Euclidean sphere $\mathbb S^{m+1}$.
		\end{enumerate}
	\end{theorem}
	
	We mention that in the this paper we do not assume that the immersion $\varphi$ is an embedding. The reason is that for $c = 0$ or $c = -1$ it is known that item a) in Theorem \ref{th:FetcuLoubeauOniciucConstantScalarCurvature} holds for any embedding $\varphi : M^m \to \mathbb R^{m+1}$ with $M$ compact and having constant normalized scalar curvature (see \cite{MontielRos1991} and \cite{Ros1988}).
	
	In our paper, we first recall some basic facts about the geometry of hypersurfaces in space forms and of biconservative hypersurfaces. We also give a somehow weaker rigidity result than item a) from Theorem \ref{th:FetcuOniciuc2022} or Theorem \ref{th:FetcuLoubeauOniciucConstantScalarCurvature}, but with more relaxed hypothesis (see Proposition \ref{th:weakerRigidityResult}).
	
	In Section \ref{sec:mLess10}, we present two new proofs for Theorem \ref{th:FetcuOniciuc2022} that are shorter and simpler than the original one. The first one is based on the Gauss equation and the fact that $(\lambda_i - \lambda_j)^2 R_{ijij} = 0$, where $\lambda_i$'s are the principal curvatures of the shape operator. The second proof exploits the facts that $|A|^2 + m^2 f^2 / 2$ is constant, $R_{1a1a} = 0$ and the properties of the connection forms. Additionally, the second proof can be generalized to prove a similar known result in the case of biconservative submanifolds with parallel normalized mean curvature vector field (PNMC). Inspired by \cite{Li1996} and \cite{Li1997}, we also extend Theorem \ref{th:FetcuOniciuc2022} by replacing the hypothesis $\Riem ^M \geq 0$ with some estimates of $|A|^2$ in terms of the mean curvature function or the normalized scalar curvature (see Theorem \ref{th:ClassificationHypersurfacesInequalityMeanCurvature} and Theorem \ref{th:ClassificationHypersurfacesInequalityScalarCurvature}). 
	
	In the last section, we give an alternative proof for Theorem \ref{th:FetcuLoubeauOniciucConstantScalarCurvature}. Further, as in the previous section, we replace the hypothesis of $\Riem ^M \geq 0$ with the same estimates of $|A|^2$ in terms of the normalized scalar curvature and obtain another classification result for compact biconservative hypersurfaces (see Theorem \ref{th:ClassificationHypersurfacesInequalityScalarCurvature2}).
	
	\textbf{Conventions.}
	In this paper, we assume that all manifolds are connected, oriented and with dimension at least $2$. In general, the metric on an arbitrary manifold is denoted by $\langle \cdot, \cdot \rangle$ or not explicitly mentioned. The Levi-Civita connection of the Riemannian manifold $M$ is denoted by $\nabla$.
	 
	The rough Laplacian acting on the set of all sections in the pull-back bundle $\varphi^{-1} (TN)$ is given by 
	$$
	\Delta^\varphi = - \trace \left ( \nabla ^\varphi \nabla ^\varphi - \nabla ^\varphi _\nabla \right ),
	$$
	where $\nabla ^\varphi$ is the connection in the pull-back bundle, and the curvature tensor field is 
	$$
	R (X, Y) Z = \left [ \nabla _X, \nabla _Y \right ] Z - \nabla _{[X, Y]} Z.
	$$
	The mean curvature function of a hypersurface $\varphi : M^m \to N^{m+1}$ is denoted by $f = (\trace A) / m$, where $A = A_\eta$ is the shape operator of $M$ and $\eta$ is a unit section in the normal bundle.
	
	If it is not explicitly mentioned otherwise, we use the Einstein summation convention.
	
	\section{Preliminaries}
	
	First, we present some general results and formulas for hypersurfaces in space forms as well as the Cheng-Yau operator. We recall the Gauss and Codazzi equations for a hypersurface $M^m$ in a space form $N^{m+1} (c)$. 
	\begin{itemize}
		\item The Gauss Equation is
		\begin{equation}\label{GaussEquation}
			R(X, Y) Z = c \left ( \left \langle Y, Z \right \rangle X - \left \langle X, Z \right \rangle Y \right ) + \left \langle A(Y), Z \right \rangle A(X) - \left \langle A(X), Z \right \rangle A(Y)
		\end{equation}
		for any $X, Y, Z \in C(TM)$.
		
		\item The Codazzi Equation is
		\begin{equation} \label{CodazziEquation}
			\left ( \nabla _X A\right ) (Y) = \left ( \nabla _Y A \right ) (X)
		\end{equation}
		for any $X, Y \in C(TM)$.
	\end{itemize}
	
	The scalar curvature of $M$ is given by 
	\begin{equation} \label{eq:formulaScalarCurvature}
		s = m (m - 1) c + m^2 f^2 - |A|^2
	\end{equation}	
	and one can define the normalized scalar curvature of $M^m$ as 
	$$
	R = \frac s {m (m - 1)}.
	$$ 
	To simplify the computations, we denote $\overline R = R - c$. Thus, we have
	\begin{equation} \label{eq:formulaNormalizedScalarCurvature}
		m (m - 1) \overline R = m^2 f^2 - |A|^2.
	\end{equation}
	
	Now we recall some well-known properties of the shape operator of hypersurfaces in space forms.
	
	\begin{lemma}\label{th:PropertiesNablaA}
		Let $\varphi : M^m \to  N^{m+1}(c)$ be hypersurface in a space form. Then
		\begin{enumerate}[label = \alph*)]
			\item $(\nabla A) (\cdot, \cdot)$ is symmetric,
			\item $\left \langle (\nabla A) (\cdot, \cdot), (\cdot) \right \rangle$ is totally symmetric,
			\item $\trace (\nabla A) (\cdot, \cdot) = m \grad f$.
		\end{enumerate}
	\end{lemma}
	
	Let $\varphi : M^m \to N^{m+1} (c)$ be a hypersurface in a space form. We denote by $\{ \lambda_i \}_{i \in \overline {1, m}}$ the principal curvatures of $M$, that is the eigenvalue functions of the shape operator. In general, these functions are continuous on $M$, but not necessarily smooth. Consider now the subset $M_A$ of all points of $M$ at which the number of distinct principal curvatures is locally constant. It is known that $M_A$ is open and dense in $M$. On a connected component of $M_A$, which is an open subset of $M$, the principal curvatures are smooth functions and there exists a local orthonormal frame field $\{E_i\}_{i \in \overline {1, m}}$ such that $A(E_i) = \lambda_i E_i$, for any $i \in \overline {1, m}$. When we consider the distinct principal curvatures of $A$, we denote their multiplicities by $m_1, \ldots, m_\ell$, where $\ell$ is the number of distinct principal curvatures.
	
	The connection $1$-forms $\omega^k _j$ corresponding to $\nabla$ are defined by $\nabla _{E_i} E_j = \omega^k _j (E_i) E_k$.
		
	In the pursuit of understanding the properties of biconservative hypersurfaces, we refer to the following formula which holds for any hypersurface in a space form $\varphi : M^m \to N^{m+1} (c)$
	\begin{equation}\label{eq:GeneralFormulaSubmanifolds}
		\frac 1 2 \Delta |A|^2 = - |\nabla A|^2 - m \Div (A(\grad f)) + m^2 |\grad f|^2 - \frac 1 2 \sum _{i, j = 1} ^m (\lambda_i - \lambda_j)^2 R_{ijij},
	\end{equation}
	see, for example, \cite{FetcuOniciuc2022}. This formula can be proved either by directly computing $\Delta |A|^2$, or as an application of the Cheng-Yau formula
	\begin{equation} \label{eq:GeneralEquationYau}
		\frac 1 2 \Delta |S|^2 = - |\nabla S|^2 - \langle S, \Hess (\trace S) \rangle - \frac 1 2 \sum _{i, j = 1} ^m (\mu_i - \mu_j)^2 R_{ijij},
	\end{equation}
	where $S$ is a symmetric $(1, 1)$ tensor field on an arbitrary manifold $M$ which satisfies the "Codazzi equation", i.e. $(\nabla S) (X, Y) = (\nabla S) (Y, X)$, and $\mu_i$'s are the eigenvalues of $S$, see \cite{ChengYau1977}.
	
	Another valuable tool in achieving our goals on biconservative hypersurfaces is the Cheng-Yau operator $\square$ associated to a symmetric $(1, 1)$ tensor field $T$, introduced in \cite{ChengYau1977}. For any $\alpha \in C^\infty (M)$, $\square \alpha$ is defined by
	\begin{equation*}
		\square \alpha = \langle T, \Hess \alpha \rangle.
	\end{equation*} 
	It was proved that if $T$ is a divergence-free tensor field defined on a compact manifold, then $\square$ is self-adjoint with respect to the $L^2$ inner product, i.e.
	\begin{equation*}
		\int _M \alpha (\square \beta) \ v_g = \int _M \beta (\square \alpha) \ v_g.
	\end{equation*}
	A direct consequence is that on compact manifolds 
	\begin{equation} \label{eq:consequenceChengYauSelfAdjoint}
		\int _M \square \alpha \ v_g = 0,
	\end{equation}
	for any divergence-free $(1, 1)$ tensor $T$.
	
	Now, we recall a few things about the stress-bienergy tensor $S_2$. Let $\phi : M^m \to \left ( N^n, h \right )$ be a smooth map, where $h$ is a Riemannian metric on $N$. Assume that $M$ is compact and on the set of all Riemannian metrics $g$ defined on $M$, one can define a new functional 
	\begin{equation*}
		\tilde E_2 (g) = \frac 1 2 \int _M h( \tau (\phi), \tau (\phi) )\ v_g,
	\end{equation*}
	where $\tau (\phi) = \trace_g \nabla d\phi$. The critical points of this functional are characterized by the vanishing of the stress-bienergy tensor $S_2$, see \cite{LoubeauMontaldoOniciuc2008}, where
	\begin{equation*}
		S_2 (X, Y) = \frac 12 |\tau (\phi)|^2 \langle X, Y \rangle + \langle d \phi, \nabla \tau (\phi) \rangle \langle X, Y \rangle - \langle d\phi (X), \nabla _Y \tau (\phi) \rangle - \langle d\phi (Y), \nabla _X \tau (\phi) \rangle,
	\end{equation*}
	for any $X, Y \in C(TM)$. The tensor field $S_2$ satisfies
	\begin{equation*}
		\Div S_2 = - \langle \tau_2 (\phi), d\phi \rangle,
	\end{equation*}
	see \cite{Jiang1986-2}.
	
	The biconservative submanifolds are defined by $\Div S_2 = 0$, which is equivalent to \eqref{eq:tau2tangentVanishing}. In the case of hypersurfaces, the stress-bienergy tensor is given by
	\begin{equation*}
		S_2 = - \frac {m^2} 2 f^2 \id + 2mf A.
	\end{equation*}
	
	The next lemma gives us another characterization for biconservative hypersurfaces.
	
	\begin{lemma} \label{th:ChengYauOperatorSelfAdjoint}
		Let $M^m$ be a hypersurface in a space form $N^{m+1} (c)$. Then $\Div S_2 = 0$ if and only if $\Div \left ( f^2 A \right ) = 0$.
	\end{lemma}
	
	\begin{proof}
		By direct computations, on $M$, we obtain
		\begin{equation*}
			\Div \left ( f^2 A \right ) = 2 f \left ( A (\grad f) + \frac m 2 f \grad f \right )
		\end{equation*}
		and 
		\begin{equation*}
			\Div S_2 = 2m \left ( A (\grad f) + \frac m 2 f \grad f \right ).
		\end{equation*}
		If $\Div S_2 = 0$, then $\Div \left ( f^2 A \right ) = 0$.
		
		Suppose that $\Div \left ( f^2 A \right ) = 0$. Then, at any point, 
		\begin{equation*}
			f = 0 \quad \text{or} \quad A (\grad f) + \frac m 2 f \grad f = 0.
		\end{equation*}
		If $A (\grad f) + mf / 2 \grad f = 0$ on $M$, then $\Div S_2 = 0$.
		
		Assume by way of contradiction that $A (\grad f) + m f / 2 \grad f \neq 0$ at a point $p \in M$, and thus at any point of an open neighbourhood $V$ of $p$. Then $f = 0$ on $V$, which implies that $\grad f = 0$ on $V$. Thus, $A (\grad f) + mf / 2 \grad f = 0$ on $V$, contradiction.
		
		Therefore, $\Div \left ( f^2 A \right ) = 0$ implies $\Div S_2 = 0$.
	\end{proof}
	
	From the previous lemma we deduce that for biconservative hypersurfaces the operator $T = f^2 A$ is divergence-free and this leads to an important formula stated in the next lemma. As you will see, this formula will be helpful in the study of biconservative hypersurfaces with constant normalized scalar curvature.
	
	\begin{lemma}
		Let $\varphi : M^m \to N^{m+1} (c)$ be a compact biconservative hypersurface. We consider the $(1, 1)$ tensor field defined by $T = f^2 A$ and its associated Cheng-Yau operator. Then
		\begin{equation} \label{eq:formulaChengYauOperatorf2A}
			0 = \int_M f^2 \left ( \frac 1 2 \Delta |A|^2 + |\nabla A|^2 + \frac 12 \sum _{i, j = 1} ^m (\lambda_i - \lambda_j)^2 R_{ijij} \right ) \ v_g.
		\end{equation}
	\end{lemma}
	
	\begin{proof}
		In our case, the operator $\square$ is defined by
		\begin{equation*}
			\square \alpha = \langle f^2 A, \Hess \alpha \rangle = f^2 \langle A, \Hess \alpha \rangle,
		\end{equation*}
		for any $\alpha \in C^\infty (M)$.
		
		Using \eqref{eq:GeneralEquationYau}, for $S = A$, we obtain that
		\begin{equation*}
			\square (mf) = - f^2 \left ( \frac 1 2 \Delta |A|^2 + |\nabla A|^2 + \frac 12 \sum _{i, j = 1} ^m (\lambda_i - \lambda_j)^2 R_{ijij} \right ).
		\end{equation*}
		Since $M$ is biconservative, Lemma \ref{th:ChengYauOperatorSelfAdjoint} implies that $\square$ is self-adjoint. Taking into account the fact that $M$ is compact, from \eqref{eq:consequenceChengYauSelfAdjoint}, we obtain $\int _M \square (mf) \ v_g = 0$ and the conclusion follows.
	\end{proof}
	
	Now we recall that the biconservative hypersurfaces in space forms are characterized by \eqref{eq:BiconservativityCharacterization}. In the following lemma we present some valuable properties of the connection forms in the case of biconservative hypersurfaces in space forms. Many of them were already proved (see, for example, \cite{BalmusMontaldoOniciuc2008}, \cite{BibiLoubeauOniciuc2021}, \cite{FuHong2018} and \cite{TurgayUpadhyay2019}), but for the sake of completeness we present their proofs here. We can work on a connected component of $M_A$, but in order to simplify the statement of the lemma we work on the whole $M$.
	
	\begin{lemma} \label{th:LemmaBiconservativeHypersurfacesConnectionForms}
		Let $\varphi : M^m \to N^{m+1} (c)$ be a biconservative hypersurface in a space form. Assume that $\grad f \neq 0$ at any point of $M$ and the number of distinct principal curvatures is constant. Then we have
		\begin{enumerate}[label = \alph*), itemsep=5pt]
			\item $\lambda_1 = - \frac m 2 f$ is a principal curvature of $A$, globally defined, with multiplicity $m_1 = 1$ and $E_1 = \grad f / |\grad f|$, \label{eq:a}
			\item $\nabla _{E_1} E_1 = 0$, \label{eq:b}
			\item $\omega^1_a (E_b) = 0$, for any $a, b \in \overline {2, m}$, $a \neq b$, \label{eq:f}
			\item $E_a(E_1(f)) = 0$, for any $a \in \overline {2, m}$, \label{eq:d}
			\item $\displaystyle \omega^a_1 (E_a) = \frac {E_1 (\lambda_a)} {\lambda_1 - \lambda_a}$, \label{eq:e}
			\item $\displaystyle R_{1a1a} = \frac {- (\lambda_1 - \lambda_a) E_1(E_1(\lambda_a)) + E_1(\lambda_1) E_1(\lambda_a) - 2 (E_1(\lambda_a))^2} {(\lambda_1 - \lambda_a)^2}$, \label{eq:g}
			\item $d \omega^1 = 0$, where by $\{\omega^i\}_{i \in \overline {1, m}}$ we denote the dual frame field of $\{E_i\}_{i \in \overline {1, m}}$. \label{eq:h}
		\end{enumerate}
	\end{lemma}
	
	\begin{proof}
		Before proving this lemma we derive some useful relation from the Codazzi equation. We have
		\begin{align*}
						   & \left ( \nabla_{E_i} A \right ) (E_j) = \left ( \nabla _{E_j} A \right ) (E_i) \\
			\Leftrightarrow& \nabla _{E_i} A(E_j) - A \left ( \nabla _{E_i} E_j \right ) = \nabla _{E_j} A(E_i) - A \left ( \nabla _{E_j} E_i \right )\\
			\Leftrightarrow& \nabla _{E_i} (\lambda_j E_j) - A \left ( \omega^\ell _j (E_i) E_\ell \right ) = \nabla _{E_j} (\lambda_i E_i) - A \left ( \omega ^\ell _i (E_j) E_\ell \right )\\
			\Leftrightarrow& (E_i (\lambda_j)) E_j + \lambda_j \nabla _{E_i} E_j - \omega^\ell _j (E_i) \lambda_\ell E_\ell = (E_j (\lambda_i)) E_i + \lambda_i \nabla _{E_j} E_i - \omega^\ell _i (E_j) \lambda_\ell E_\ell\\
			\Leftrightarrow& (E_i (\lambda_j)) E_j + \lambda_j \omega^\ell _j (E_i) E_\ell - \lambda_\ell \omega^\ell _j (E_i) E_\ell = (E_j (\lambda_i)) E_i + \lambda_i \omega^\ell _i (E_j) E_\ell - \lambda_\ell \omega^\ell _i (E_j) E_\ell\\
			\Leftrightarrow& (E_i (\lambda_j)) E_j + \sum _{\ell = 1} ^m (\lambda_j - \lambda_\ell) \omega^\ell _j (E_i) E_\ell = (E_j (\lambda_i)) E_i + \sum _{\ell = 1} ^m (\lambda_i - \lambda_\ell) \omega^\ell _i (E_j) E_\ell.
		\end{align*}
		Taking into account that $\{E_i\}_{i \in \overline {1, m}}$ is a local orthonormal frame field, we obtain that
		\begin{align}
			& E_i (\lambda_j) = (\lambda_i - \lambda_j) \omega^j _i (E_j), \label{eq:Codazzi1}\\
			& (\lambda_j - \lambda_\ell) \omega^\ell _j (E_i) = (\lambda_i - \lambda_\ell) \omega^\ell _i (E_j), \label{eq:Codazzi2}.
		\end{align}
		for any mutually distinct $i, j, \ell \in \overline {1, m}$. In these relations we do not use the Einstein summation convention.
		
		In order to prove \ref{eq:a}, note that the biconservativity condition \eqref{eq:BiconservativityCharacterization} implies that $- mf / 2$ is a principal curvature of $A$. We can assume that $\lambda_1 = - mf / 2$ and $E_1 = \grad f / |\grad f |$. It follows that $E_a (f) = 0$, for any $a \in \overline {2, m}$. 
		
		Now, suppose by way of contradiction that the multiplicity $m_1$ is at least $2$, i.e. there is $\alpha \in \overline {2, m}$ such that $\lambda_\alpha = \lambda_1$. From \eqref{eq:Codazzi1}, for $i = 1$ and $j = \alpha$, we get that $E_1 (\lambda_1) = 0$, which implies that $E_1 (f) = 0$, contradiction. Thus, $m_1 = 1$.
		
		To prove \ref{eq:b}, we replace $i = a$ and $j = 1$ in \eqref{eq:Codazzi1} which gives us $E_a (\lambda_1) = (\lambda_a - \lambda_1) \omega^1 _a (E_1)$, thus $\omega^1 _a (E_1) = 0$, i.e. $\nabla _{E_1} E_1 = 0$.
		
		In order to show \ref{eq:f}, we first compute $[E_a, E_b] (f)$ in two ways. From the definition of the Lie bracket, we have
		$$
		[E_a, E_b] (f) = E_a (E_b (f)) - E_b (E_a (f)) = 0.
		$$
		On the other hand, 
		\begin{align*}
			[E_a, E_b] (f) =& \left ( \nabla _{E_a} E_b - \nabla _{E_b} E_a \right ) (f)\\
						   =& \left ( \omega^i _b (E_a) - \omega^i _a (E_b) \right ) E_i (f)\\
						   =& \left ( \omega^1 _b (E_a) - \omega^1 _a (E_b) \right ) E_1 (f).
		\end{align*}
		Thus, for any $a, b \in \overline {2, m}$, we obtain $\omega^1 _b (E_a) = \omega^1 _a (E_b)$.
		
		Further, we consider two cases. First, if $\lambda_a = \lambda_b$, then we consider $i = 1$, $j = b$ and $\ell = a$ in \eqref{eq:Codazzi2} and obtain $(\lambda_1 - \lambda_a) \omega^a _1 (E_b) = 0$, which implies that $\omega^1 _a (E_b) = 0$, for any distinct $a, b \in \overline {2, m}$.
		
		Then, if $\lambda_a \neq \lambda_b$, we replace $i = a$, $j = b$ and $\ell = 1$ in \eqref{eq:Codazzi2} and using the fact that $\omega^1 _b (E_a) = \omega^1 _a (E_b)$ we obtain $(\lambda_b - \lambda_a) \omega^1 _a (E_b) = 0$, which implies that $\omega^1 _a (E_b) = 0$, for any distinct $a, b \in \overline {2, m}$. Therefore, relation \ref{eq:f} is proved.
		
		Similarly, relation \ref{eq:d} can be proved by computing in two ways $[E_1, E_a] (f)$, for any $a \in \overline {2, m}$.
		
		Relation \ref{eq:e} is obtained by replacing $i = 1$ and $j = a$ in \eqref{eq:Codazzi1}.
		
		Using \ref{eq:f} we obtain 
		\begin{align*}
			 & R_{1a1a} = R_{a1a1} = \langle R(E_a, E_1) E_1, E_a \rangle\\
			=& \left \langle \nabla _{E_a} \nabla _{E_1} E_1 - \nabla _{E_1} \nabla _{E_a} E_1 - \nabla _{[E_a, E_1]} E_1, E_a \right \rangle \\
			=& - \left \langle \nabla _{E_1} \left ( \omega^i _1 (E_a) E_i \right ) + \nabla _{\nabla _{E_a} E_1 - \nabla _{E_1} E_a} E_1, E_a \right \rangle \\
			=& - \left \langle \left ( E_1 \left ( \omega^i _1 (E_a) \right ) \right ) E_i + \omega^i _1 (E_a) \omega^j _i (E_1) E_j + \omega^i _1 (E_a) \nabla _{E_i} E_1 - \omega^i _a (E_1) \nabla _{E_i} E_1, E_a \right \rangle \\
			=& - E_1 \left ( \omega^a _1 (E_a) \right ) - \omega^i _1 (E_a) \omega^a _i (E_1) - \omega^i _1 (E_a) \omega^a _1 (E_i) + \omega^i _a (E_1) \omega^a _1 (E_i) \\
			=& E_1 \left ( \omega^1 _a (E_a) \right ) + \sum _{b = 2} ^m \omega^1 _b (E_a) \omega^a _b (E_1) - \sum _{b = 2} ^m \omega^1 _b (E_a) \omega^1 _a (E_b) - \sum _{b = 2} ^m \omega^b _a (E_1) \omega^1 _a (E_b) \\
			=& E_1 \left ( \omega^1 _a (E_a) \right ) - \left ( \omega^1 _a (E_a) \right )^2.
		\end{align*}
		Taking into account \ref{eq:e} we obtain \ref{eq:g}.
		
		Now we prove \ref{eq:h}. By definition 
		\begin{equation*}
			\left ( d \omega^1 \right ) (X, Y) = X \left ( \omega^1 Y \right ) - Y \left ( \omega^1 X \right ) - \omega^1 ([X, Y]).
		\end{equation*}
		Since $d \omega^1$ is bilinear and antisymmetric, it is sufficient to prove that 
		\begin{equation*}
			d \omega^1 (E_1, E_a) = 0,\quad \forall a \in \overline {2, m} \quad \text{and} \quad d \omega^1 (E_a, E_b) = 0, \quad \forall a, b \in \overline {2, m},\ a \neq b.
		\end{equation*}
		Using \ref{eq:b} and \ref{eq:f} we obtain
		\begin{align*}
			d \omega^1 (E_1, E_a) =& E_1 \left ( \omega^1 (E_a) \right ) - E_a \left ( \omega^1 (E_1) \right ) - \omega^1 ([E_1, E_a])\\
								  =& - \omega^1 \left ( \nabla _{E_1} E_a - \nabla _{E_a} E_1 \right ) \\
								  =& - \left ( \omega^\ell _a (E_1) - \omega^\ell _1 (E_a) \right ) \omega^1 (E_\ell) \\
								  =& - \left ( \omega^1 _a (E_1) - \omega^1 _1 (E_a) \right ) = 0
		\end{align*}
		and 
		\begin{align*}
			d \omega^1 (E_a, E_b) =& E_a \left ( \omega^1 (E_b) \right ) - E_b \left ( \omega^1 (E_a) \right ) - \omega^1 ([E_a, E_b]) \\
								  =& - \omega^1 \left ( \nabla _{E_a} E_b - \nabla _{E_b} E_a \right ) \\
								  =& - \left ( \omega^\ell _b (E_a) - \omega^\ell _a (E_b) \right ) \omega^1 (E_\ell) \\
								  =& - \left ( \omega^1 _b (E_a) - \omega^1 _a (E_b) \right ) = 0.
		\end{align*}
		Therefore, relation \ref{eq:h} is proved.
	\end{proof}
	
	From formula \eqref{eq:BiconservativityCharacterization} we quickly derive a rigidity result for biconservative hypersurfaces with $\Riem ^M 
	\geq 0$ in non-positively curved space forms. This result does not rely on the compactness of $M$, or $m \leq 10$, or on the constancy of the normalized scalar curvature.
	
	\begin{proposition} \label{th:weakerRigidityResult}
		Let $\varphi : M^m \to N^{m+1} (c)$ be a biconservative hypersurface in a space form with $c \leq 0$. If $\Riem ^M \geq 0$, then $M$ is CMC.
	\end{proposition}
	
	\begin{proof}
		Assume by way of contradiction that $\grad f \neq 0$ at any point of an open subset $V$. On $V$ we can assume that $\lambda_1 = -mf / 2$.
		
		From the hypothesis $R_{1a1a} \geq 0$, for any $a \in \overline {2, m}$ and summing up, we get
		\begin{align*}
			 \sum _{a = 2} ^m R_{1a1a} \geq 0 \Leftrightarrow& \sum _{a = 2} ^m (c + \lambda_1 \lambda_a) \geq 0 \\
			\Leftrightarrow& (m - 1)c + \lambda_1 (mf - \lambda_1 ) \geq 0\\
			\Leftrightarrow& (m - 1) c \geq \frac {3 m^2} 4 f^2.
		\end{align*}
		As $c \leq 0$ and since $f$ cannot vanish on $V$, we get a contradiction.
	\end{proof}
	
	\begin{remark}
		If we additionally assume in Proposition \ref{th:weakerRigidityResult} that $M$ is compact, using Theorem \ref{th:NomizuSmythTheorem}, we obtain a) from Theorem \ref{th:FetcuOniciuc2022} or Theorem \ref{th:FetcuLoubeauOniciucConstantScalarCurvature}.
	\end{remark}
	
	The biconservativity proves to be very useful in handling the general formula \eqref{eq:GeneralFormulaSubmanifolds}. Indeed, using the characterization of biconservativity \eqref{eq:BiconservativityCharacterization} and its consequence \eqref{eq:JHChenInequality}, we directly obtain the the following important inequality.
	\begin{lemma}
		Let $\varphi : M^m \to N^{m+1} (c)$ be a biconservative hypersurface in a space form. Then 
		\begin{equation} \label{eq:MainInequality}
			\frac 1 2 \Delta \left ( |A|^2 + \frac {m^2} 2 f^2 \right ) \leq \frac {3(m - 10)} {m + 26} |\nabla A|^2 - \frac 12 \sum _{i, j = 1} ^m (\lambda_i - \lambda_j)^2 R_{ijij}.
		\end{equation}
	\end{lemma}
	
	At the end of this section, we recall the Okumura Lemma.
	\begin{lemma} [\cite{Okumura1974}] \label{th:OkumuraInequality}
		Let $b_1, \ldots, b_m$ be real numbers such that $\sum _{i = 1} ^m b_i = 0$. Then 
		\begin{equation*}
			- \frac {m - 2} {\sqrt {m (m - 1)}} \left ( \sum _{i = 1} ^m b_i ^2 \right )^{\frac 3 2} \leq \sum _{i = 1} ^m b_i^3 \leq \frac {m - 2} {\sqrt {m (m - 1)}} \left ( \sum _{i = 1} ^m b_i^2 \right ) ^{\frac 3 2}.
		\end{equation*}
		Moreover, equality holds in the right-hand (respectively, left-hand) side if and only if $(m - 1)$ of the $b_i$'s are non-positive (respectively, non-negative) and equal.
	\end{lemma}
	
	\section{Biconservative hypersurfaces with dimension at most $10$} \label{sec:mLess10}
	
	The proof of Theorem \ref{th:FetcuOniciuc2022} was not provided explicitly in \cite{FetcuOniciuc2022} but, as indicated by the authors, it can be established by following the same steps used in the proof of Theorem \ref{th:FetcuLoubeauOniciuc2021}, which is quite long and laborious. It relies on utilizing the equality condition in \eqref{eq:JHChenInequality} and the fact that
	\begin{equation} \label{eq:TAZero}
		0 = - \frac 1 2 \sum _{i, j = 1} ^{m} (\lambda_i - \lambda_j)^2 R_{ijij} = \langle - \trace (RA) ( \cdot, \ , \cdot), A \rangle.
	\end{equation}
	From the equality condition in \eqref{eq:JHChenInequality}, one obtains the expressions of $(\nabla A) (E_i, E_j)$, where $\{E_i\}_{i \in \overline {1, m}}$ is a local orthonormal frame field which diagonalizes the shape operator. Finally, replacing these in 
	$$
	\sum _{i, j = 1} ^m \left \langle \left ( \nabla ^2 A \right ) (E_j, E_i, E_i) - \left ( \nabla ^2 A \right ) (E_i, E_j, E_i), A (E_j) \right \rangle = 0,
	$$
	which follows from \eqref{eq:TAZero}, one obtains that $f$ is constant. Therefore, we can apply the result from \cite{NomizuSmyth1969} and conclude.
	
	Next we present two alternative proofs for Theorem \ref{th:FetcuOniciuc2022} and two generalizations of it. Also, in this section we present another proof of a known result concerning biconservative PNMC submanifolds in space forms.
	
	\subsection{The first alternative proof of Theorem \ref{th:FetcuOniciuc2022}}
	
	Since $m \leq 10$ and $\Riem ^M \geq 0$, the right-hand side of \eqref{eq:MainInequality} is non-positive, i.e.
	\begin{equation*}
		\frac 12 \Delta \left ( |A|^2 + \frac {m^2} 2 f^2 \right ) \leq 0.
	\end{equation*}
	Since $M$ is compact, the last relation implies that $|A|^2 + m^2 f^2 / 2$ is constant on $M$, i.e. the left-hand side of \eqref{eq:MainInequality} vanishes on $M$. Therefore we obtain that on $M$ we have 
	\begin{equation*}
		\sum _{i, j = 1} ^m (\lambda_i - \lambda_j)^2 R_{ijij} = 0.
	\end{equation*}
	Using again the fact that $\Riem ^M \geq 0$, we obtain 
	\begin{equation*}
		(\lambda_i - \lambda_j)^2 R_{ijij} = 0, \quad \forall i, j \in \overline {1, m}.
	\end{equation*}
	Since, from the Gauss equation, for any distinct $i, j \in \overline {1, m}$, we have $R_{ijij} = c + \lambda_i \lambda_j$, we deduce that
	\begin{equation*}
		(\lambda_i - \lambda_j)^2 (c + \lambda_i \lambda_j) = 0, \quad \forall i, j \in \overline {1, m}.
	\end{equation*}
	In fact, on $M$, we obtain
	\begin{equation} \label{eq:RelationPrincipalCurvatures}
		(\lambda_i - \lambda_j) (c + \lambda_i \lambda_j) = 0, \quad \forall i, j \in \overline {1, m}.
	\end{equation}
	The last relation implies that $M$ has at most two distinct principal curvatures at any point of $M$.
	
	Consider now the subset $M_A$ of all point in $M$ at which the number of distinct principal curvatures is locally constant. In the following we will show that $\grad f = 0$ on every connected component of $M_A$ and thus, from density, we will conclude that $\grad f = 0$ on $M$, i.e. $f$ is constant. Then, from \eqref{eq:GeneralFormulaSubmanifolds} we get that $\nabla A = 0$.
	
	We choose an arbitrary connected component of $M_A$. Since $M$ has at most two distinct principal curvatures, on this component we have: either each of its points is umbilical, or each of its points has exactly two distinct principal curvatures. For simplicity, we denote by $M$ the chosen connected component.
	
	If $M$ is umbilical, then it is CMC.
	
	We suppose now that $M$ has exactly two distinct principal curvatures at any point of $M$. In this case, $A$ is locally diagonalizable with respect to a smooth orthonormal frame field $\{E_i\}_{i \in \overline {1, m}}$. Denote
	\begin{equation*}
		\lambda_1 = \ldots = \lambda_{m_1} \quad \text{and} \quad \lambda_{m_1 + 1} = \ldots = \lambda_m.
	\end{equation*}
	
	Assume that $\grad f \neq 0$ and we will obtain a contradiction. If necessary, we can restrict ourselves to an open subset of $M$, denoted again by $M$, such that $\grad f \neq 0$ at any point of $M$.
	
	Now, using \ref{eq:a} from Lemma \ref{th:LemmaBiconservativeHypersurfacesConnectionForms} we can assume that
	\begin{equation*}
		\lambda_1 = - \frac m 2 f, \quad m_1 = 1 \quad \text{and} \quad E_1 = \frac {\grad f} {|\grad f|}
	\end{equation*}
	on $M$. Since $\trace A = mf$, we have
	\begin{equation*}
		\lambda_2 = \frac {3m} {2(m - 1)} f.
	\end{equation*}
	Using \eqref{eq:RelationPrincipalCurvatures}, we obtain
	\begin{equation*}
		0 = c + \lambda_1 \lambda_2 = c - \frac {3m^2} {4(m - 1)} f^2.
	\end{equation*}
	Note that this relation is false if $c \leq 0$. If $c > 0$, since $f$ is smooth, we obtain that $f$ is constant on $M$, which contradicts $\grad f \neq 0$ at any point of $M$. \hfill \qed
	
	\subsection{The second alternative proof of Theorem \ref{th:FetcuOniciuc2022}} \label{subsec:SecondProof}
	
	As in the first proof, using \eqref{eq:MainInequality}, we obtain that $|A|^2 + m^2 f^2 / 2$ is constant on $M$ and $(\lambda_i - \lambda_j)^2 R_{ijij} =0$ on $M$.
	
	We will prove that $\grad f = 0$ on every connected component of $M_A$ and thus, from density, we will obtain that $\grad f = 0$ on $M$. We choose an arbitrary component of $M_A$, which is also denoted by $M$. Assume that $\grad f \neq 0$ at any point of $M$ and we will obtain a contradiction. So, on $M$, the principal curvatures $\lambda_i$'s are smooth, $i \in \overline {1, m}$, and $A$ is locally smoothly diagonalizable with respect to a local orthonormal frame field $\{E_i\}_{i \in \overline {1, m}}$.
	
	Using \ref{eq:a} from Lemma \ref{th:LemmaBiconservativeHypersurfacesConnectionForms}, we can assume that
	\begin{equation*}
		\lambda_1 = - \frac m 2 f, \quad m_1 = 1 \quad \text{and} \quad E_1 = \frac {\grad f} {|\grad f|}
	\end{equation*}
	on $M$.
	
	Since $m_1 = 1$ and $(\lambda_i - \lambda_j)^2 R_{ijij} =0$, we have $R_{1a1a} = 0$, for any $a \in \overline {2, m}$. Using this and \ref{eq:g} from Lemma \ref{th:LemmaBiconservativeHypersurfacesConnectionForms}, we obtain 
	\begin{equation} \label{eq:RelationE1E1}
		(\lambda_1 - \lambda_a) E_1 (E_1 (\lambda_a)) = E_1 (\lambda_1) E_1 (\lambda_a) - 2 (E_1 (\lambda_a))^2.
	\end{equation} 
	We have
	\begin{align*}
		|A|^2 + \frac {m^2} 2 f^2 =& \lambda_1^2 + \sum _{a = 2} ^m \lambda_a^2 + \frac {m^2} 2 f^2 = \frac {m^2} 4 f^2 + \sum _{a = 2} ^m \lambda_a^2 + \frac {m^2} 2 f^2 \\
								  =& \frac {3m^2} 4 f^2 + \sum _{a = 2} ^m \lambda_a ^2.
	\end{align*}
	Since $|A|^2 + m^2 f^2 / 2$ is constant, we obtain
	\begin{equation} \label{eq:RelationFromA2}
		\frac {3m^2} 4 f E_1 (f) + \sum _{a = 2} ^m \lambda_a E_1 (\lambda_a) = 0.
	\end{equation}
	On the other hand, from the fact that $\trace A = mf$, we deduce
	\begin{equation} \label{eq:RelationFromTraceA}
		\sum _{a = 2} ^m E_1 (\lambda_a) = \frac {3m} 2 E_1 (f).
	\end{equation}
	Replacing \eqref{eq:RelationFromTraceA} in \eqref{eq:RelationFromA2}, we obtain
	\begin{equation*}
		\frac m 2 f \sum _{a = 2} ^m E_1 (\lambda_a) + \sum _{a = 2} ^m \lambda_a E_1 (\lambda_a) = 0,
	\end{equation*}
	which is equivalent to 
	\begin{equation*}
		\sum _{a = 2} ^m (\lambda_1 - \lambda_a) E_1 (\lambda_a) = 0.
	\end{equation*}
	Taking the derivative of the above relation with respect to $E_1$ and using \eqref{eq:RelationE1E1}, we get
	\begin{align*}
		& \sum _{a = 2} ^m \left ( (E_1 (\lambda_1) - E_1 (\lambda_a)) E_1 (\lambda_a) + (\lambda_1 - \lambda_a) E_1 (E_1 (\lambda_a)) \right ) = 0\\
		\Leftrightarrow& \sum _{a = 2} ^m \left ( E_1 (\lambda_1) E_1 (\lambda_a) - (E_1 (\lambda_a))^2 + E_1 (\lambda_1) E_1 (\lambda_a) - 2 (E_1 (\lambda_a))^2 \right ) = 0\\
		\Leftrightarrow& \sum _{a = 2} ^m \left ( 2 E_1 (\lambda_1) E_1 (\lambda_a) - 3 (E_1 (\lambda_a))^2 \right ) = 0\\
		\Leftrightarrow& -m E_1 (f) \sum _{a = 2} ^m E_1 (\lambda_a) - 3 \sum _{a = 2} ^m (E_1 (\lambda_a))^2 = 0\\
		\Leftrightarrow& -m E_1 (f) \frac {3m} 2 E_1 (f) - 3 \sum _{a = 2} ^m (E_1 (\lambda_a))^2 = 0\\
		\Leftrightarrow& - \frac {3m^2} 2 (E_1 (f))^2 - 3 \sum _{a = 2} ^m (E_1 (\lambda_a))^2 = 0\\
		\Leftrightarrow& \frac {m^2} 2 |\grad f|^2 + \sum _{a = 2} ^m (E_1 (\lambda_a))^2 = 0.
	\end{align*}
	Therefore, $|\grad f| = 0$, which contradicts the assumption that $\grad f \neq 0$. \hfill \qed
	
	\subsection{The case of PNMC biconservative submanifolds}
	
	In this subsection, we first recall the definition of PNMC submanifolds and some well-known results concerning them. After this, we will provide a new proof of a known rigidity result for PNMC submanifolds in space forms.
	
	Let $\varphi : M^m \to N^n$ be a submanifold such that its mean curvature vector field $H$ does not vanish at any point of $M$. In this case, we denote by $h = |H| > 0$ and by $\eta_0 = H / |H|$ a unit normal vector field with the same direction as $H$. If $\eta_0$ is parallel in the normal bundle, i.e. $\nabla ^\perp \eta_0 = 0$, we say that $M$ has parallel normalized vector field and we call $M$ a PNMC submanifold. We denote by $A_0$ the shape operator of $M$ in the direction of $\eta_0$. As in the case of hypersurfaces, we can define the set $M_{A_0}$ of all points of $M$ at which the number of distinct eigenvalues of $A_0$ is locally constant. Also, in this case, $M_{A_0}$ is open and dense in $M$. On a connected component of $M_{A_0}$, the eigenvalues of $A_0$ are smooth functions and $A_0$ is (locally) diagonalizable and, in this case, we also denote the local orthonormal frame field which diagonalizes $A_0$ by $\{E_i\}_{i \in \overline {1, m}}$. The eigenvalues of $A_0$ are denoted by $\lambda_i$. The multiplicities of the distinct eigenvalues of $A_0$ are denoted by $m_1, \ldots, m_\ell$. 
	
	Again, the connection $1$-forms corresponding to $\nabla$ are defined by $\nabla _{E_i} E_j = \omega^k _j (E_i) E_k$.
	
	As in the case of hypersurfaces, we briefly present some well-known properties of the shape operator $A_0$.
	
	\begin{lemma} [\cite{FetcuLoubeauOniciuc2021}]
		Let $\varphi : M^m \to N^n (c)$ be a PNMC submanifold in a space form. Then
		\begin{enumerate}[label = \alph*)] 
			\item $\nabla A_0$ is symmetric,
			\item $\langle (\nabla A_0) (\cdot, \cdot), \cdot \rangle$ is totally symmetric,
			\item $\trace A_0 = m h$,
			\item $\Div A_0 = \trace (\nabla A_0) = m \grad h$.
		\end{enumerate}
	\end{lemma}
	
	One can prove that biconservative PNMC submanifolds are characterized by 
	\begin{equation} \label{eq:BiconservativityCharacterizationPNMC}
		A_0 (\grad h) = -\frac m 2 h \grad h,
	\end{equation}
	see \cite{FetcuLoubeauOniciuc2021}.
	
	The following lemma for PNMC biconservative submanifolds in space forms provides us similar properties of the connection forms as in the case of hypersurfaces. The proof of this lemma follows the same steps used in the proof of Lemma \ref{th:LemmaBiconservativeHypersurfacesConnectionForms}.
	
	\begin{lemma}
		Let $\varphi : M^m \to N^n (c)$ be a biconservative $PNMC$ submanifold in a space form. Assume that $\grad h \neq 0$ at any point of $M$ and the number of distinct principal curvatures of $A_0$ is constant. Then we have
		\begin{enumerate}[label = \alph*), itemsep=5pt]
			\item $\lambda_1 = - \frac m 2 h$ is a principal curvature of $A_0$, globally defined, with multiplicity $m_1 = 1$ and $E_1 = \grad h / |\grad h|$, \label{eq:aPNMC}
			\item $\nabla _{E_1} E_1 = 0$, \label{eq:bPNMC}
			\item $\omega^1_a (E_b) = 0$, for any $a, b \in \overline {2, m}$, $a \neq b$, \label{eq:fPNMC}
			\item $E_a(E_1(f)) = 0$, for any $a \in \overline {2, m}$, \label{eq:dPNMC}
			\item $\displaystyle \omega^a_1 (E_a) = \frac {E_1 (\lambda_a)} {\lambda_1 - \lambda_a}$, \label{eq:ePNMC}
			\item $\displaystyle R_{1a1a} = \frac {- (\lambda_1 - \lambda_a) E_1(E_1(\lambda_a)) + E_1(\lambda_1) E_1(\lambda_a) - 2 (E_1(\lambda_a))^2} {(\lambda_1 - \lambda_a)^2}$, \label{eq:gPNMC}
			\item $d \omega^1 = 0$, where by $\{\omega^i\}_{i \in \overline {1, m}}$ we denote the dual frame field of $\{E_i\}_{i \in \overline {1, m}}$. \label{eq:hPNMC}
		\end{enumerate}
	\end{lemma}
	
	As in the case of hypersurfaces, one can prove the following lemma.
	\begin{lemma} [\cite{FetcuLoubeauOniciuc2021}] \label{th:JHChenInequalityPNMC}
		Let $\varphi : M^m \to N^n (c)$ be a PNMC biconservative submanifold in a space form. Then 
		\begin{equation*}
			|\nabla A_0|^2 \geq \frac {m^2 (m + 26)} {4(m - 1)} |\grad h|^2.
		\end{equation*}
	\end{lemma}
	
	Using this Lemma \ref{th:JHChenInequalityPNMC}, \eqref{eq:GeneralEquationYau} and \eqref{eq:BiconservativityCharacterizationPNMC}, we obtain 
	\begin{lemma}
		Let $\varphi : M^m \to N^n (c)$ be a PNMC biconservative submanifold in a space form. Then
		\begin{equation*}
			\frac 1 2 \Delta \left ( |A_0|^2 + \frac {m^2} 2 h^2 \right ) \leq \frac {3(m - 10)} {m + 26} |\nabla A_0|^2 - \frac 1 2 \sum _{i, j = 1} ^m (\lambda_i - \lambda_j)^2 R_{ijij},
		\end{equation*}
		where $\lambda_i$'s are the eigenvalues of $A_0$.
	\end{lemma}
	
	Now, following the steps of the second proof \ref{subsec:SecondProof}, we obtain an alternative proof of 
	\begin{theorem}[\cite{FetcuLoubeauOniciuc2021}] \label{th:FetcuLoubeauOniciuc2021}
		Let $\varphi : M^m \to N^n (c)$ be a compact PNMC biconservative submanifold in a space form with $\Riem ^M \geq 0$ and $m \leq 10$. Then $M$ has parallel mean curvature vector field and $\nabla A_H = 0$.
	\end{theorem}
	
	\subsection{Generalizations of Theorem \ref{th:FetcuOniciuc2022}}
	
	First we present a generalization of Theorem \ref{th:FetcuOniciuc2022} replacing the hypothesis $\Riem ^M \geq 0$ with some estimates of $|A|^2$ is terms of the mean curvature function $f$.
	
	\begin{theorem} \label{th:ClassificationHypersurfacesInequalityMeanCurvature}
		Let $\varphi : M^m \to N^{m+1} (c)$ be a compact non-minimal biconservative hypersurface. If $m \leq 10$, $c + f^2 \geq 0$ when $c < 0$ and 
		\begin{equation} \label{eq:hypothesisMeanCurvature}
			mf^2 \leq |A|^2 \leq mc + \frac {m^3} {2(m-1)} f^2 - \frac {m(m-2)} {2(m-1)} \sqrt { m^2 f^4 + 4 (m-1) cf^2},
		\end{equation}
		then $\nabla A = 0$. Moreover, either $|A|^2 = mf^2$ on $M$, i.e. $M$ is umbilical, or we have equality in the second inequality of \eqref{eq:hypothesisMeanCurvature}. More precisely, we have
		\begin{enumerate}[label = \alph*)]
			\item if $c = -1$, then $\varphi (M)$ is a hypersphere of $\mathbb S^m (r)$ of radius $r > 0$, 
			\item if $c = 0$, then $\varphi (M)$ is a hypersphere $\mathbb S^m (r)$
			\item if $c = 1$, then $\varphi (M)$ is either a small hypersphere $\mathbb S^m (r)$ of $\mathbb S^{m+1}$, or the standard product $\mathbb S^1 (r_1) \times \mathbb S^{m-1} (r_2)$, where $r_1^2 + r_2^2 = 1$ and $r_1 > 1 / {\sqrt m}$.			
		\end{enumerate}
	\end{theorem}
	
	\begin{proof}
		Let $p \in M$ be an arbitrary point and let $\{ e_i \}_{i \in \overline {1, m}}$ be an orthonormal basis in $T_p M$ such that $A (e_i) = \lambda_i e_i$, for any $i \in \overline {1, m}$. From the Gauss equation we have
		\begin{align*}
			\frac 12 \sum _{i, j = 1} ^m (\lambda_i - \lambda_j)^2 R_{ijij} =& \frac 12 \sum _{i, j = 1} ^m \left ( \lambda_i^2 - 2 \lambda_i \lambda_j + \lambda_j^2 \right ) (c + \lambda_i \lambda_j)\\
			=& \frac 12 \sum _{i, j = 1} ^m \left ( c\lambda_i^2 + \lambda_i^3 \lambda_j - 2c \lambda_i \lambda_j - 2 \lambda_i^2 \lambda_j^2 + c\lambda_j^2 + \lambda_i \lambda_j^3 \right )\\
			=& \frac 1 2 \sum _{j = 1} ^m \left ( c |A|^2 + \lambda_j \trace A^3 - 2mcf \lambda_j - 2 \lambda_j^2 |A|^2 + mc \lambda_j^2 + mf \lambda_j^3 \right ) \\
			=& \frac 1 2 \left ( mc |A|^2 + mf \trace A^3 - 2m^2c f^2 - 2|A|^4 + mc |A|^2 + mf \trace A^3 \right ) \\
			=& mc |A|^2 - m^2 c f^2 - |A|^4 + mf \trace A^3.
		\end{align*}
		Now, consider $\mu_i = \lambda_i - f$. We have
		\begin{equation*}
			\sum _{i = 1} ^m \mu_i = 0, \quad \sum _{i = 1} ^m \mu_i ^2 = |A|^2 - mf^2 \quad \text{and} \quad \trace A^3 = \sum _{i = 1} ^m \mu_i^3 + 3f |A|^2 - 2mf^3.
		\end{equation*}
		Thus 
		\begin{equation} \label{eq:intermediaryFormOfSumWithRijij}
			\frac 1 2 \sum _{i, j = 1} ^m (\lambda_i - \lambda_j)^2 R_{ijij} = \left ( |A|^2 - m f^2 \right ) \left ( mc - |A|^2 + 2mf^2 \right ) + mf \sum _{i = 1} ^m \mu_i ^3.
		\end{equation}
		Since $\sum _{i=1} ^m \mu_i = 0$, we can apply Lemma \ref{th:OkumuraInequality} and obtain
		\begin{equation} \label{eq:OkumuraInequalityMu}
			\left | \sum _{i=1} ^m \mu_i^3 \right | \leq \frac {m-2} {\sqrt {m (m-1)}} \left ( \sum _{i=1} ^m \mu_i^2 \right )^{\frac 3 2}.
		\end{equation}
		Therefore, equation \eqref{eq:intermediaryFormOfSumWithRijij} implies that 
		\begin{align*}
			\frac 1 2 \sum _{i, j = 1} ^m (\lambda_i - \lambda_j)^2 R_{ijij} \geq& \left ( |A|^2 - mf^2 \right ) \left ( mc - |A|^2 + 2mf^2 \right ) - m |f| \left | \sum _{i = 1} ^m \mu_i ^3 \right |\\
			\geq& \left ( |A|^2 - mf^2 \right ) \left ( mc - |A|^2 + 2mf^2 \right ) \\
			& - \frac {m (m - 2)} {\sqrt {m (m - 1)}} |f| \left ( |A|^2 - mf^2 \right ) \sqrt {|A|^2 - mf^2}
		\end{align*}
		and we conclude that
		\begin{equation} \label{eq:inequalitySumWithRijij}
			- \frac 12 \sum _{i, j = 1} ^m (\lambda_i - \lambda_j)^2 R_{ijij} \leq \left (|A|^2 - mf^2 \right ) \left ( |A|^2 - mc - 2mf^2 + \frac {m (m - 2)} {\sqrt {m (m - 1)}} |f| \sqrt {|A|^2 - mf^2} \right )
		\end{equation}
		Further,
		\begin{align}
			& |A|^2 - mc - 2mf^2 + \frac {m (m - 2)} {\sqrt {m (m - 1)}} |f| \sqrt {|A|^2 - mf^2} \label{eq:SecondParanthesisWithSquare} \\
			=& \left ( |A|^2 - mf^2 + 2 \frac {m (m - 2)} {2 \sqrt {m (m - 1)}} |f| \sqrt {|A|^2 - mf^2} + \frac {m^2 (m - 2)^2} {4 m (m - 1)} f^2 \right ) \notag \\
			& - \frac {m^2 (m - 2)^2} {m (m - 1)} f^2 - mc - mf^2 \notag \\
			=& \left ( \sqrt {|A|^2 - mf^2} + \frac {m (m - 2)} {2 \sqrt {m (m - 1)}} |f| \right )^2 - m \left ( c + \frac {m^2} {4(m - 1)} f^2 \right ) \notag
		\end{align}
		Using \eqref{eq:MainInequality}, \eqref{eq:inequalitySumWithRijij} and \eqref{eq:SecondParanthesisWithSquare}, we obtain 
		\begin{align}
			\frac 1 2 \Delta \left ( |A|^2 + \frac {m^2} 2 f^2 \right ) \leq& \frac {3 (m - 10)} {m + 26} |\nabla A|^2 + \left (|A|^2 - mf^2 \right ) \left [ \left ( \sqrt {|A|^2 - mf^2} + \frac {m (m - 2)} {2 \sqrt {m (m - 1)}} |f| \right )^2 \right .  \label{eq:inequalityLaplacianMeancurvature} \\
			&  \left . - m \left ( c + \frac {m^2} {4(m - 1)} f^2 \right ) \right ]. \notag
		\end{align}
		Now we prove that
		\begin{equation} \label{eq:intermediaryInequalityMeanCurvature}
			\left ( \sqrt {|A|^2 - mf^2} + \frac {m (m - 2)} {2 \sqrt {m (m - 1)}} |f| \right )^2 - m \left ( c + \frac {m^2} {4(m - 1)} f^2 \right ) \leq 0.
		\end{equation}
		is equivalent to the second inequality of \eqref{eq:hypothesisMeanCurvature}. Indeed, since $c + f^2 \geq 0$, we obtain that $c + (m^2 f^2) / (4 (m - 1)) \geq 0$ and \eqref{eq:intermediaryInequalityMeanCurvature} is equivalent to 
		\begin{align*}
			& \sqrt {|A|^2 - mf^2} + \frac {m (m - 2)} {2 \sqrt {m (m - 1)}} |f| \leq \sqrt {mc + \frac {m^3} {4 (m - 1)} f^2}\\
			\Leftrightarrow& \sqrt {|A|^2 - mf^2} \leq \sqrt {mc + \frac {m^3} {4 (m - 1)} f^2} - \frac {m (m - 2)} {2 \sqrt {m (m - 1)}} |f|
		\end{align*}
		The fact that $c + f^2 \geq 0$, implies that the right-hand side is non-negative and the inequality is equivalent to 
		\begin{align*}
			|A|^2 - mf^2 \leq& \left ( \sqrt {mc + \frac {m^3} {4 (m - 1)} f^2} - \frac {m (m - 2)} {2 \sqrt {m (m - 1)}} |f| \right )^2\\
			=& mc + \frac {m^3} {4 (m - 1)} f^2 - \frac {m (m - 2)} {\sqrt {m (m - 1)}} |f| \sqrt {mc + \frac {m^3} {4 (m - 1)} f^2} + \frac {m^2 (m - 2)^2} {4 m (m - 1)} f^2,						 
		\end{align*}
		which represents the second inequality of \eqref{eq:hypothesisMeanCurvature}.
		
		Now, using \eqref{eq:inequalityLaplacianMeancurvature} and \eqref{eq:intermediaryInequalityMeanCurvature} we obtain that
		\begin{equation*}
			\frac 1 2 \Delta \left ( |A|^2 + \frac {m^2} 2 f^2 \right ) \leq 0.
		\end{equation*}
		Since $M$ is compact, we obtain that 
		\begin{equation*}
			\frac 1 2 \Delta \left ( |A|^2 + \frac {m^2} 2 f^2 \right ) = 0.
		\end{equation*}
		Therefore, on $M$, we have equality in \eqref{eq:MainInequality} which means equality in \eqref{eq:JHChenInequality}, in either the first or the second inequality of \eqref{eq:hypothesisMeanCurvature}, and in the Okumura's inequality \eqref{eq:OkumuraInequalityMu}.
		
		We note that, in fact, inequality \eqref{eq:OkumuraInequalityMu} is multiplied by $|f|$, and thus we must have equality in the Okumura's inequality on the set $W$ of all points where $f$ does not vanish. For the simplicity of notation, we can assume that $f$ does not vanish at any point of $M$. Otherwise, we can work on (a connected component of) $W$ and prove that $\grad f = 0$. Then it follows that $\grad f = 0$ on $M$.
		
		From the equality in \eqref{eq:JHChenInequality}, on $M$, we obtain 
		\begin{equation} \label{eq:equalityJHChen}
			|\nabla A|^2 = \frac {m^2 (m + 26)} {4 (m - 1)} |\grad f|^2.
		\end{equation}
		From equality in Okumura's inequality, we obtain that at any point of $M$ either 
		$$
		\mu_1 = \ldots = \mu_m = 0,
		$$
		or
		$$
		\mu_1 \neq \mu_2 = \ldots = \mu_m.
		$$
		Using the fact that $\mu_i = \lambda_i - f$, for any $i \in \overline {1, m}$, we obtain that at any point either
		$$
		\lambda_1 = \ldots = \lambda_m = f,
		$$
		or 
		$$
		\lambda_1 \neq \lambda_2 = \ldots = \lambda_m,
		$$
		which means that at any point of $M$ we have one or two distinct principal curvatures. Thus, on any connected component $U$ of $M_A$, we either have one distinct principal curvature, or two distinct principal curvatures.
		
		From the equality in \eqref{eq:hypothesisMeanCurvature}, we obtain that at any point of $M$ either
		$$
		|A|^2 = mf^2,
		$$
		or
		$$
		|A|^2 = mc + \frac {m^3} {2 (m - 1)} f^2 - \frac {m (m - 2)} {2 (m - 1)} \sqrt {m^2 f^4 + 4 (m - 1) cf^2}.
		$$
		On $U$, if we have one distinct principal curvature, then $f$ is constant, i.e. $\grad f = 0$ on $U$.
		
		Suppose now that we have two distinct principal curvatures on $U$. Assume by way of contradiction that $\grad f \neq 0$ at any point of an open subset $V$ of $U$.
		
		Since $M$ is biconservative, we can assume that $\lambda_1 = - mf / 2$. Now, using the fact that $\trace A = mf$, we obtain that $\lambda_2 = (3m f) / (2 (m - 1))$. Thus 
		$$
		|A|^2 = \frac {m^2 (m + 8)} {4 (m - 1)} f^2.
		$$
		Since $|A|^2 = m f^2$ is equivalent to the fact that $V$ has one distinct principal curvature at any point, we obtain that
		\begin{align*}
			& |A|^2 = mc + \frac {m^3} {2 (m - 1)} f^2 - \frac {m (m - 2)} {2 (m - 1)} \sqrt {m^2 f^4 + 4 (m - 1) c f^2}\\
			\Leftrightarrow& \frac {(m - 2)} {2 (m - 1)} \sqrt {m^2 f^4 + 4 (m - 1) c f^2} = c + \frac {m (m - 8)} {2 (m - 1)} f^2. 
		\end{align*}
		If $c + (m (m - 8)) / (2 (m - 1)) f^2 < 0$, then we have a contradiction, since the left-hand side is non-negative.
		
		If $c + (m (m - 8)) / (2 (m - 1)) f^2 \geq 0$, then, by squaring, we obtain a polynomial equation with constant coefficients in the variable $f$. Thus $f$ is constant, which contradicts $\grad f \neq 0$ at any point of $V$.
		
		Therefore, $f$ must be constant on $U$, thus on $M_A$. Since $M_A$ is dense in $M$, we obtain that $\grad f = 0$ on $M$, i.e. $f$ is constant on $M$.
		
		Using \eqref{eq:equalityJHChen}, we obtain that $\nabla A = 0$ on $M$, thus $M$ is either umbilical, or it has two distinct principal curvatures.
		
		In order to obtain the explicit classification of our biconservative hypersurfaces, it is enough to consider only the cases $c = -1$, $c = 0$ and $c = 1$.
		
		Consider $c = -1$. Since $M$ is compact, the only solution to our problem must be umbilical and, moreover, $\varphi (M)$ is a hypersphere $\mathbb S^m (r)$ of radius $r > 0$. We note that the hypothesis $c + f^2 \geq 0$ does not give any restriction for the radius $r$.
		
		Now, we consider $c = 0$. Again, since $M$ is compact, the only solution must be umbilical and $\varphi (M)$ is a hypersphere $\mathbb S^m (r)$.
		
		Finally, consider $c = 1$. If $M$ is umbilical, then $\varphi (M)$ is a small hypersphere $\mathbb S^m (r)$ of $\mathbb S^{m+1}$.
		
		If $M$ has two distinct principal curvatures, then $\varphi (M)$ is the standard product $\mathbb S^1 (r_1) \times \mathbb S^{m - 1} (r_2)$, where $r_1^2 + r_2^2 = 1$. Considering 
		$$
		\eta_p = \left ( \frac {r_2} {r_1} p_1, - \frac {r_1} {r_2} p_2 \right ),
		$$
		where $p = (p_1, p_2) \in \mathbb S^1 (r_1) \times \mathbb S^{m - 1} (r_2)$, we obtain 
		$$
		\lambda_1 = - \frac {r_2} {r_1} \quad \text{and} \quad \lambda_2 = \frac {r_1} {r_2}.
		$$
		Thus 
		$$
		|A|^2 = \frac {r_2^2} {r_1^2} + (m - 1) \frac {r_1^2} {r_2^2}
		$$
		and 
		$$
		f^2 = \frac 1 {m^2} \left ( (m - 1)^2 \frac {r_1^2} {r_2^2} - 2 (m - 1) + \frac {r_2^2} {r_1^2} \right ).
		$$
		We know that 
		\begin{align*}
			& |A|^2 = mc + \frac {m^3} {2 (m - 1)} f^2 - \frac {m (m - 2)} {2 (m - 1)} \sqrt {m^2 f^4 + 4 (m - 1) cf^2} \\
			\Leftrightarrow& \frac {r_2^2} {r_1^2} + (m - 1) \frac {r_1^2} {r_2^2} = m + \frac m {2 (m - 1)} \left ( (m - 1)^2 \frac {r_1^2} {r_2^2} - 2 (m - 1) + \frac {r_2^2} {r_1^2} \right ) \\
			& - \frac {m - 2} {2 (m - 1)} \left | (m - 1) \frac {r_1} {r_2} - \frac {r_2} {r_1} \right | \sqrt {(m - 1)^2 \frac {r_1^2} {r_2^2} - 2 (m - 1) + \frac {r_2^2} {r_1^2} + 4 (m - 1)} \\
			\Leftrightarrow& (m - 1)^2 \frac {r_1^2} {r_2^2} - \frac {r_2^2} {r_1^2} = \left |(m - 1)^2 \frac {r_1^2} {r_2^2} - \frac {r_2^2} {r_1^2} \right |
		\end{align*}
		Thus 
		\begin{equation*}
			(m - 1)^2 \frac {r_1^2} {r_2^2} - \frac {r_2^2} {r_1^2} \geq 0 \Leftrightarrow r_1 \geq \frac 1 {\sqrt m}.
		\end{equation*}
		Since $M$ is not minimal, we conclude that $\varphi (M)$ is $\mathbb S^1 (r_1) \times \mathbb S^{m-1} (r_2)$, where $r_1^2 + r_2^2 = 1$ and $r_1 > 1 / \sqrt m$.
	\end{proof}	
	
	The following result is similar to Theorem \ref{th:ClassificationHypersurfacesInequalityMeanCurvature}, but the estimates of $|A|^2$ are in terms of the normalized scalar curvature. Since the proof is similar to that of Theorem \ref{th:ClassificationHypersurfacesInequalityMeanCurvature}, we underline only its key points.
	
	\begin{theorem} \label{th:ClassificationHypersurfacesInequalityScalarCurvature}
		Let $\varphi : M^m \to N^{m+1} (c)$ be a compact non-minimal biconservative hypersurface. If $3 \leq m \leq 10$, $m \overline R + 2c > 0$, 
		\begin{equation*}
			\overline R \geq \frac {m-2} {2m(m-1)} |A|^2 - \frac m {2(m-1)} c
		\end{equation*}
		and 
		\begin{equation} \label{eq:hypothesisScalarCurvature}
			m \overline R \leq |A|^2 \leq \frac m {(m - 2) (m \overline R + 2c)} \left ( mc^2 + 4(m - 1) c \overline R + m (m - 1) \overline R^2 \right ),
		\end{equation}
		then $\nabla A = 0$. Moreover, either $|A|^2 = m \overline R$ on $M$, i.e. $M$ is umbilical, or we have equality in the second inequality of \eqref{eq:hypothesisScalarCurvature}. More precisely, we have
		\begin{enumerate}[label = \alph*)]
			\item if $c = -1$, then $\varphi (M)$ is a hypersphere of $\mathbb S^m (r)$ of radius $r > 0$, 
			\item if $c = 0$, then $\varphi (M)$ is a hypersphere $\mathbb S^m (r)$
			\item if $c = 1$, then $\varphi (M)$ is either a small hypersphere $\mathbb S^m (r)$ of $\mathbb S^{m+1}$, or the standard product $\mathbb S^1 (r_1) \times \mathbb S^{m-1} (r_2)$, where $r_1^2 + r_2^2 = 1$ and $r_1 > 1 / {\sqrt m}$.			
		\end{enumerate}
	\end{theorem}
	
	\begin{proof}
		From the proof of Theorem \ref{th:ClassificationHypersurfacesInequalityMeanCurvature}, we have formula \eqref{eq:inequalityLaplacianMeancurvature}. Taking into account \eqref{eq:formulaNormalizedScalarCurvature} and replacing in \eqref{eq:inequalityLaplacianMeancurvature}, we get
		\begin{align*}
			\frac 3 4 \Delta \left ( |A|^2 + \frac {m(m - 1)} 3 \overline R \right ) \leq& \frac {3(m - 10)} {m + 26} |\nabla A|^2 + \frac {m - 1} m \left ( |A|^2 - m \overline R \right ) \left ( \frac {m - 2} m |A|^2 - mc \right . \\
			& \left . - 2 (m - 1) \overline R + (m - 2) \sqrt {\frac 1 {m^2} |A|^4 + \frac {m - 2} m \overline R |A|^2 - (m - 1) \overline R^2} \right )
		\end{align*}
		Following the same steps as in the proof of Theorem \ref{th:ClassificationHypersurfacesInequalityMeanCurvature}, we obtain equality in \eqref{eq:JHChenInequality} on $M$. Also, at any point of $M$, we have either
		$$
		\lambda_1 = \ldots = \lambda_m = f,
		$$
		or
		$$
		\lambda_1 \neq \lambda_2 = \ldots = \lambda_m.
		$$
		and either
		$$
		m \overline R = |A|^2,
		$$
		or
		\begin{equation} \label{eq:equalityInTheSecondPartOfDoubleInequalityScalarCurvature}
			|A|^2 = \frac m {(m - 2) (m \overline R + 2c)} \left ( mc^2 + 4 (m - 1) \overline R + m (m - 1) \overline R^2 \right ).
		\end{equation}
		Thus, on any connected component $U$ of $M_A$, we either have one distinct principal curvature, or two distinct principal curvatures.
		
		If on $U$ we have one distinct principal curvature, then $\grad f = 0$, on $U$.
		
		If on $U$ we have two distinct principal curvatures, then \eqref{eq:equalityInTheSecondPartOfDoubleInequalityScalarCurvature} must hold. Assume by way of contradiction that $\grad f \neq 0$ at any point of an open neighbourhood $V$. Since $M$ is biconservative, on $V$ there exists a local orthonormal frame field $\{ E_i \}_{i \in \overline {1, m}}$ such that $A (E_i) = \lambda_i E_i$, $\lambda_1 = - mf / 2$ and $E_1 = \grad f / |\grad f|$. Using the fact that $\trace A = mf$, we obtain that
		$$
		\lambda_2 = \frac {3m} {2(m - 1)} f.
		$$
		Computing $|A|^2$ and using \eqref{eq:formulaNormalizedScalarCurvature}, we obtain that 
		\begin{equation*}
			|A|^2 = \frac {m^2 (m + 8)} {4 (m - 1)} f^2 = \frac {m + 8} {4 (m - 1)} |A|^2 + \frac {m (m + 8)} 4 \overline R,
		\end{equation*}
		which is equivalent to 
		\begin{equation} \label{eq:SecondFormulaNormA}
			\frac {3 (m - 4)} {m - 1} |A|^2 = m (m + 8) \overline R.
		\end{equation}
		Combining \eqref{eq:equalityInTheSecondPartOfDoubleInequalityScalarCurvature} and \eqref{eq:SecondFormulaNormA}, we obtain a non-trivial polynomial equation in the variable $\overline R$ with constant coefficients, which implies that $\overline R$ is constant on $V$. From \eqref{eq:equalityInTheSecondPartOfDoubleInequalityScalarCurvature} we deduce that $|A|^2$ is constant and together with \eqref{eq:formulaNormalizedScalarCurvature}, we conclude that $f$ is constant on $V$, i.e. $\grad f = 0$ on $V$, contradiction.
		
		Therefore, $\grad f = 0$ on $M_A$, and thus on $M$.
		
		Now, applying the same steps as in the proof of Theorem \ref{th:ClassificationHypersurfacesInequalityMeanCurvature}, the conclusion follows.
	\end{proof}
	
	\begin{remark}
		When $\varphi (M) = \mathbb S^1 (r_1) \times \mathbb S^{m-1} (r_2)$, for any $r_1 \in \left ( 1 / \sqrt m, \sqrt {2 / m} \right )$ we have $\overline R < 0$.
	\end{remark}
	
	\section{Biconservative hypersurfaces with constant normalized scalar curvature}
	
	The original proof of Theorem \ref{th:FetcuLoubeauOniciucConstantScalarCurvature} is based on a Simons type formula obtained by computing $(1 / 2) \Delta |S_2|^2$. Here we provide a simpler proof involving Cheng-Yau operator.
	
	\subsection{Alternative proof of Theorem \ref{th:FetcuLoubeauOniciucConstantScalarCurvature}}
	
	Since $M$ has constant normalized scalar curvature, from \eqref{eq:formulaScalarCurvature}, we obtain $\Delta |A|^2 = m^2 \Delta f^2$. Further, by straightforward computations, we deduce that $f^2 \Delta f^2 = (1 / 2) \Delta f^4 + 4 f^2 |\grad f|^2$, which implies that 
	\begin{equation*}
		\frac 1 2 f^2 \Delta |A|^2 = \frac {m^2} 2 \left ( \frac 1 2 \Delta f^4 + 4 f^2 |\grad f|^2 \right ).
	\end{equation*}
	Since $M$ is a compact biconservative hypersurface, combining the above formula with \eqref{eq:formulaChengYauOperatorf2A}, we obtain
	\begin{equation*}
		0 = \int_M  f^2 \left ( 2m^2 |\grad f|^2 + |\nabla A|^2 + \frac {1} 2 \sum _{i, j = 1} ^m (\lambda_i - \lambda_j)^2 R_{ijij} \right ) \ v_g.
	\end{equation*}
	Since $\Riem ^M \geq 0$, we obtain that, on $M$,
	\begin{equation*}
		f^2 |\grad f|^2 = 0 \quad \text{and} \quad f^2 |\nabla A|^2 = 0.
	\end{equation*} 
	
	If $\grad f = 0$ on $M$, then $f$ is constant and non-zero, thus $\nabla A = 0$ on $M$.
	
	Assume by way of contradiction that $\grad f \neq 0$ at a point $p \in M$, thus on an open neighbourhood $U$ of $p$. Then $f = 0$ on $U$, which implies that $\grad f = 0$ on $U$, contradiction.
	
	In conclusion, $\grad f = 0$ on $M$, i.e. $M$ is CMC, and thus $\nabla A = 0$. \hfill \qed
	
	\subsection{Generalization of Theorem \ref{th:FetcuLoubeauOniciucConstantScalarCurvature}}
	
	As in the previous section, we replace the hypothesis of $\Riem ^M \geq 0$ with some estimates of $|A|^2$ and obtain
	
	\begin{theorem} \label{th:ClassificationHypersurfacesInequalityScalarCurvature2}
		Let $M^m$ be a compact non-minimal biconservative hypersurface in $N^{m+1} (c)$ with constant normalized scalar curvature, $m \geq 3$, $m \overline R + 2c > 0$, 
		\begin{equation*}
			\overline R \geq \frac {m-2} {2m(m-1)} |A|^2 - \frac m {2(m-1)} c
		\end{equation*}
		and 
		\begin{equation} \label{eq:hypothesisScalarCurvature2}
			m \overline R \leq |A|^2 \leq \frac m {(m - 2) (m \overline R + 2c)} \left ( mc^2 + 4(m - 1) c \overline R + m (m - 1) \overline R^2 \right ),
		\end{equation}
		Then $\nabla A = 0$. Moreover, either $|A|^2 = m \overline R$ on $M$, i.e. $M$ is umbilical, or we have equality in the second inequality of \eqref{eq:hypothesisScalarCurvature2}. More precisely, we have
		\begin{enumerate}[label = \alph*)]
			\item if $c = -1$, then $\varphi (M)$ is a hypersphere of $\mathbb S^m (r)$ of radius $r > 0$, 
			\item if $c = 0$, then $\varphi (M)$ is a hypersphere $\mathbb S^m (r)$
			\item if $c = 1$, then $\varphi (M)$ is either a small hypersphere $\mathbb S^m (r)$ of $\mathbb S^{m+1}$, or the standard product $\mathbb S^1 (r_1) \times \mathbb S^{m-1} (r_2)$, where $r_1^2 + r_2^2 = 1$ and $r_1 > 1 / {\sqrt m}$.			
		\end{enumerate}
	\end{theorem}
	
	\begin{remark}
		We mention that, for the sake of completeness, we kept in Theorem \ref{th:ClassificationHypersurfacesInequalityScalarCurvature2} the case $c = 0$. In fact, when $c = 0$, the result follows without assuming the biconservativity of the hypersurface, but only the constancy of the normalized scalar curvature and \eqref{eq:hypothesisScalarCurvature2} (see Theorem 3 in \cite{Li1996}). However, with the additional hypothesis of biconservativity the proof of the case $c = 0$ is much shorter and simpler the original one.
	\end{remark}
	
	\textbf{Acknowledgements} We are grateful to Yu Fu, Eric Loubeau and Cezar Oniciuc for their valuable suggestions and discussions.

	\bibliographystyle{abbrv}
	\bibliography{Bibliography.bib}
\end{document}